\documentclass{amsart}
\usepackage{comment}
\usepackage{amscd}
\usepackage{amssymb}
\usepackage[all, knot]{xy}
\xyoption{all}
\xyoption{arc}
\usepackage{hyperref}

\def\height{\operatorname{height}}
\def\ann{\operatorname{ann}}

\def\cube#1#2#3#4#5#6#7#8{
& #5 \ar[rr] \ar[dl] \ar@{-}[d] && #6 \ar[dd] \ar[dl] \\
#1 \ar[rr] \ar[dd]  & \ar[d] & #2 \ar[dd] \\
& #7 \ar@{-}[r] \ar[dl] & \ar[r] & #8 \ar[dl] \\
#3 \ar[rr] && #4 \\
}

\def\b{\beta}

\def\ker{\operatorname{ker}}

\def\L{\Lambda}

\def\dm{\operatorname{dim}}
\def\rank{\operatorname{rank}}

\def\Tor{\operatorname{Tor}}

\def\Spec{\operatorname{Spec}}

\def\fl{{\mathsf{fl}}}

\newcommand{\Z}{\mathbb{Z}}

\newcommand{\fp}{{\mathfrak p}}
\newcommand{\fm}{{\mathfrak m}}

\numberwithin{equation}{section}

\theoremstyle{plain} 
\newtheorem{thm}[equation]{Theorem}

\newtheorem{introthm}{Theorem}

\newtheorem*{introconj*}{Conjecture}
\newtheorem*{introthm*}{Theorem}

\newtheorem{prop}[equation]{Proposition}

\theoremstyle{definition}
\newtheorem{defn}[equation]{Definition}

\theoremstyle{remark}

\def\Perf{\operatorname{Perf}}

\newcommand{\id}{\operatorname{id}}

\def\len{\operatorname{\ell}}

\def\rk{\operatorname{rank}}

\def\and{{ \text{ and } }}

\def\rnk{\operatorname{rank}}

\def\th{^{\mathrm{th}}}
\def\nd{^{\mathrm{nd}}}

\begin{document}
\begin{abstract}
The Buchsbaum-Eisenbud-Horrocks Conjecture predicts that the $i\th$ Betti number $\beta_i(M)$ of 
a non-zero module $M$ of finite length and finite projective dimension over a local
ring $R$ of dimension $d$ should be at least ${d \choose i}$. It would follow from
the validity of this conjecture that $\sum_i \beta_i(M) \geq 2^{d}$. We prove the latter inequality 
holds in a large number of cases and that, when $R$ is a complete
intersection in which $2$ is invertible, 
equality holds if and only if  $M$ is isomorphic to the quotient of $R$ by a regular sequence of elements.
\end{abstract}

\title{Total Betti numbers of modules of finite projective dimension}

\author{Mark E. Walker}
\address{Department of Mathematics, University of Nebraska, Lincoln, NE 68588-0130, USA}
\email{mark.walker@unl.edu}

\thanks{This work was partially supported by  grant  \#318705  from the Simons Foundation.}

\maketitle

\section{Introduction} 

We recall a long-standing conjecture (see \cite[1.4]{BE} and \cite[Problem 24]{Ha}):

\begin{introconj*}[Buchsbaum-Eisenbud-Horrocks Conjecture] 
Let $R$ be a commutative Noetherian ring such that $\Spec(R)$ is connected and let $M$ be a non-zero, finitely generated $R$-module of finite projective
dimension.   For any finite projective resolution 
$$
0 \to P_d \to \cdots \to P_1 \to P_0 \to M \to 0
$$
of $M$ we have
$$
\rank_R(P_i) \geq {c \choose i} 
$$
where $c = \height_R(\ann_R(M))$, the height of the annihilator ideal of $M$.
\end{introconj*}

The validity of the Buchsbaum-Eisenbud-Horrocks Conjecture would imply
that the ``total rank'' of any projective resolution of $M$ is at least $2^c$.
In this paper,  we prove this latter inequality holds in a large number of cases:

\begin{introthm} \label{introthm1}
Assume  $R$, $M$, and $P_\cdot$ are as in the Buchsbaum-Eisenbud-Horrocks Conjecture, and that in addition
\begin{enumerate}
\item 
$R$ is locally a complete intersection
 and $M$ is $2$-torsion free, or
\item $R$ contains $\Z/p$ as a subring for an odd prime $p$.
\end{enumerate}
Then $\sum_i \rank_R(P_i) \geq 2^c$, where $c = \height_R(\ann_R(M))$.
\end{introthm}

Theorem \ref{introthm2} below is the special case of Theorem \ref{introthm1} in which we assume 
$R$ is a local and $M$ has finite length.  
We record it as a separate theorem since  Theorem \ref{introthm1} follows immediately from it and also because in the local situation we
can say a bit more. 

For a local ring $R$ and a finitely generated $R$-module $M$, let $\beta_i(M)$ be the  $i\th$
Betti number of $R$, defined to be the rank of the $i\th$ free
module in the minimal free resolution of $M$.

\begin{introthm} \label{introthm2}
Assume $(R, \fm, k)$ is a local (Noetherian, commutative) ring of Krull dimension $d$ and that $M$ is a non-zero
$R$-module of finite length and finite projective dimension. 
If either
\begin{enumerate}
\item $R$ is the quotient of a regular local ring by a regular sequence of elements and $2$ is invertible in $R$, or
\item $R$ contains $\Z/p$ as a subring for an odd prime $p$,
\end{enumerate}
then $\sum_i \b_i(M) \geq 2^d$. 

Moreover, if the assumptions in (1) hold and
$\sum_i \b_i(M) = 2^d$, then $M$ is isomorphic to the quotient of $R$ by a regular sequence of $d$ elements.
\end{introthm}

To see that Theorem \ref{introthm1} follows from Theorem \ref{introthm2}, with the notation of the first theorem,  let
$\fp$ be a minimal prime containing $\ann_R(M)$ of height $c$. Then
$\dm(R_\fp) = c$, $M_\fp$ has finite length, and $\beta_i(M_\fp) \leq \rank_R(P_i)$ for all $i$. Moreover, if 
$M$ is $2$-torsion free, then $2 \notin \fp$ and hence is invertible in $R_\fp$.

I thank Seth Lindokken,  Michael Brown, Claudia Miller, Peder Thompson and Luchezar Avramov for useful conversations
about this paper.

\section{Complete intersections of residual characteristic not $2$}

In this section we prove part (1) of Theorem \ref{introthm2} and the assertion  concerning when the equation 
$\sum_i \b_i(M) = 2^d$  holds; see Theorem \ref{thm1} below.  

For any local ring $(R, \fm, k)$, let $\Perf^\fl(R)$ be the category of bounded complexes of finite rank free $R$-modules $F_\cdot$ such that $H_i(F_\cdot)$ 
has finite length for all $i$, and define $K_0^\fl(R)$ to be the Grothendieck group of $\Perf^\fl(R)$. 
Recall that $K_0^\fl(R)$ is generated by isomorphism classes of objects of $\Perf^\fl(R)$, modulo relations coming from short exact sequences and quasi-isomorphisms.

Let $\psi^2: K_0^\fl(R) \to K_0^\fl(R)$ be the $2\nd$ Adams operation, as defined by Gillet-Soul\'e \cite{GS87}. Gillet-Soul\'e's definition involves the
Dold-Kan correspondence between complexes and simplicial modules, but
if $2$ is invertible in $R$, then $\psi^2$ admits a simpler description: For $F_\cdot \in \Perf^\fl(R)$ let
$T^2(F_\cdot)$ denote its second tensor power $F_\cdot \otimes_R F_\cdot$
endowed with the action of the 
symmetric group $\Sigma_2 = \langle \tau \rangle$
given  by 
$$
\tau \cdot (x \otimes y) = (-1)^{|x||y|} y \otimes x.
$$ 
Since $\frac12 \in R$,  we have a direct sum decomposition
$T^2(F_\cdot) = S^2(F_\cdot) \oplus \L^2(F_\cdot)$, 
where $S^2(F_\cdot) := \ker(\tau - \id)$ and $\L^2(F_\cdot) := \ker(\tau + \id)$. By \cite[6.14]{BMTW} we have
\begin{equation} \label{eqn1}
\psi^2[F_\cdot] = [S^2(F_\cdot)] - [\L^2(F_\cdot)] \in K_0^\fl(R).
\end{equation}

Let $\len_R$ denote the length of an $R$-module and write
$\chi: K_0^\fl(R) \to \Z$ for the Euler characteristic map: 
$\chi\left([F_\cdot]\right) = \sum_i (-1)^i \len_R H_i(F_\cdot)$.

\begin{prop}[Gillet-Soul\'e] (See \cite[7.1]{GS87}.) \label{prop1} 
If $R$ is a local complete intersection of dimension $d$, then $\chi \circ \psi^2 = 2^d \cdot \chi$. 
\end{prop}

\begin{defn} \label{qrdefnot2}
A local ring $(R, \fm, k)$ of dimension $d$ such that $2$ is invertible in $R$ will be called a {\em quasi-Roberts ring} if
there we have an equality of maps $\chi \circ \psi^2 = 2^d \cdot \chi$.
\end{defn}

\begin{thm} \label{thm1} Let $(R, \fm, k)$ be a local ring of dimension $d$ such that $2$ is invertible in $R$. If $R$ is a
quasi-Roberts ring,  then for any non-zero $R$-module $M$ of finite length and finite projective dimension, we have 
$\sum_i \beta_i(M) \geq 2^d$. 

Moreover,
if $\sum_i \beta_i(M) = 2^d$ then $M \cong R/(y_1, \dots, y_d)$ for some regular sequence of elements $y_1, \dots, y_d \in \fm$.
\end{thm}

\begin{proof}
Let $F_\cdot$ be the minimal free resolution of $M$, so that $\chi(F_\cdot) = \len_R(M)$ and 
$\rnk_R(F_i) = \beta_i(M)$.  Using \eqref{eqn1} we get
\begin{equation} \label{E124a}
\begin{aligned}
2^d \cdot \len_R(M)    = \chi(\psi^2(F)) 
& =   \sum_i (-1)^i \len_RH_i(S^2(F_\cdot))  - \sum_j  (-1)^j \len_RH_j(\L^2(F_\cdot))  \\ 
& \leq   \sum_{\text{$i$ even}} \len_RH_i(S^2(F_\cdot))  + \sum_{\text{$i$ odd}} \len_RH_i(\L^2(F_\cdot)).  \\ 
\end{aligned}
\end{equation}
Since $S^2(F_\cdot)$ and $\L^2(F_\cdot)$ are direct summands of $F_\cdot \otimes_R F_\cdot$, 
\begin{equation}\label{E124b}
\sum_{\text{$i$ even}} \len_RH_i(S^2(F_\cdot))  + \sum_{\text{$i$ odd}} \len_RH_i(\L^2(F_\cdot))
\leq   \sum_i \len_RH_i(F_\cdot \otimes_R F_\cdot).
\end{equation}
For each $i$, $H_i(F_\cdot \otimes_R F_\cdot) \cong H_i(F_\cdot \otimes_R M)$ 
is a subquotient of $F_i \otimes_R M$ and thus
\begin{equation}\label{E124c}
\len_R H_i(F_\cdot \otimes_R M)
\leq \len_R(F_i \otimes_R M) = 
\rk(F_i) \cdot \len_R(M) = \beta_i(M) \cdot \len_R(M).
\end{equation}
Putting the inequalities \eqref{E124a}, 
\eqref{E124b}, and \eqref{E124c} together yields
$$
2^d \cdot \len_R(M) \leq  \len_R(M)  \cdot \sum_i \beta_i(M)
$$
and since $\len_R(M) > 0$, we conclude $\sum_i \beta_i(M) \geq 2^d$. 

Now suppose 
$\sum_i \beta_i(M) = 2^d$. Then the inequalities 
\eqref{E124a}, \eqref{E124b}, and \eqref{E124c} 
must actually be equalities, which means that 
$H_i(S^2(F_\cdot))  = 0$ for all odd $i$, 
$H_j(\L^2(F_\cdot))  = 0$ for all even $j$, and
$F_\cdot \otimes_R M$ has trivial differential.
Since  $H_0(\L^2(F_\cdot)) \cong \L^2(M)$ is the classical second exterior power, $M$ must be cyclic; i.e., of the form $R/I$ for some ideal $I$. 
Since $F_\cdot \otimes_R R/I$ has trivial differential, $I/I^2 \cong \Tor_1^R(R/I, R/I)$ is free as an
$R/I$-module, and thus a result of
Ferrand and Vasconcelos (see \cite[2.2.8]{BrunsHerzog}) gives that
$I$ is generated by a regular sequence of elements.
\end{proof}

\section{Rings of odd characteristic}

In this section we prove part (2) of Theorem \ref{introthm2}.
The main idea is to replace the Euler characteristic $\chi$ occurring in the proof of part (1)  with the {\em Dutta
  multiplicity}.

\begin{defn} Assume $(R, \fm, k)$ is a complete local ring of dimension $d$ that contains $\Z/p$ as a subring for some
  prime $p$ and that $k$ is a perfect field. 
For $F_\cdot \in \Perf^\fl(R)$, define
$$
\chi_\infty(F_\cdot) = \lim_{e \to \infty} \frac{\chi(\varphi^eF_\cdot)}{p^{de}}
$$
where $\varphi^e$ denotes extension of scalars along the $e\th$ iterate of the Frobenius endomorphism of $R$.
The limit is known to exist by, e.g., \cite[7.3.3]{RobertsBook}.
\end{defn}

\begin{proof}[Proof of Theorem \ref{introthm2} part (2)]
There is a faithfully flat map $(R, \fm, k)  \to (R', \fm', k')$ of local rings  
such that $\fm \cdot R' = \fm'$, $R'$ is complete, and $k'$ is algebraically closed; see \cite[0.10.3.1]{EGA}.
Letting $M' := M \otimes_R R'$, we have that $M'$  is a non-zero $R'$-module of finite length and finite projective dimension,   
$\beta_i^{R'}(M') = \beta_i^R(M)$ for all $i$,  
and $\dm(R') = \dm(R)$.  We may therefore assume $R$ is complete with algebraically closed residue field.

Let $F_\cdot$ be the minimal free resolution of $M$. 
Since $R$ is complete with perfect residue field, a result of Roberts \cite[7.3.5]{RobertsBook} gives 
\begin{equation} \label{eqn6}
\chi_\infty(F_\cdot) > 0
\end{equation}
and a result of Kurano-Roberts \cite[3.1]{KuranoRoberts} gives (using \eqref{eqn1})
\begin{equation} \label{eqn5}
\chi_\infty(S^2(F_\cdot)) - \chi_\infty(\L^2(F_\cdot)) = \chi_\infty(\psi^2(F_\cdot)) = 2^d \cdot \chi_\infty(F_\cdot).
\end{equation}
For each $e \geq 0$ we have
$\varphi^e S^2(F_\cdot) \cong S^2(\varphi^eF_\cdot)$ and $\varphi^e \L^2(F_\cdot)) \cong \L^2(\varphi^eF_\cdot)$ and thus
$$
\begin{aligned}
\chi_\infty(S^2(F_\cdot)) & = \lim_{e \to \infty} \frac{1}{p^{de}} \sum_i (-1)^i \len_R H_i(S^2(\varphi^e F_\cdot)) \\
\chi_\infty(\L^2(F_\cdot)) & = \lim_{e \to \infty} \frac{1}{p^{de}} \sum_i (-1)^i \len_R H_i(\L^2(\varphi^e F_\cdot)). \\
\end{aligned}
$$
As in the proof of Theorem \ref{thm1}, for a fixed $e$ we have
$$
\frac{1}{p^{de}} \sum_i (-1)^i \len_R H_i(S^2(\varphi^e F_\cdot))
- \frac{1}{p^{de}} \sum_i (-1)^i \len_R H_i(\L^2(\varphi^e F_\cdot)) 
\leq  \sum_j \len_R H_j(T^2(\varphi^e F_\cdot)). 
$$

By \cite[1.7]{PS}, the complex $\varphi^e(F_\cdot)$ is the minimal free resolution of the finite length module $\varphi^e(M)$, for each $e \geq 0$. 
As in the proof of Theorem \ref{thm1}, for each $i$ we have
$$
\len_R H_i(T^2(\varphi^e F_\cdot)) \leq \rnk(\varphi^e F_i) \cdot \ell_R(\varphi^e M) = \beta_i(M) \cdot \chi(\varphi^e F_\cdot).
$$

We have proven that
$$
\frac{1}{p^{de}} \sum_i (-1)^i \len_R H_i(\varphi^e S^2( F_\cdot))
- \frac{1}{p^{de}} \sum_i (-1)^i \len_R H_i(\varphi^e \L^2(F_\cdot)) 
\leq \frac{1}{p^{de}} \chi(\varphi^e F_\cdot) \cdot \sum_i \beta_i(M)
$$
holds for each $e \geq 0$. 
Taking limits and using \eqref{eqn5} gives
$$
2^d \cdot  \chi_\infty(F_\cdot) \leq \chi_\infty(F_\cdot) \cdot \sum_i \beta_i(M).
$$
Since $\chi_\infty(F_\cdot) > 0$ by \eqref{eqn6}, we conclude
$\sum_i \beta_i(M) \geq 2^d$.
\end{proof}

\bibliographystyle{plain}

\def\cprime{$'$}


\begin{thebibliography}{1}

\bibitem{BMTW}
Michael~K. Brown, Claudia Miller, Peder Thompson, and Mark~E. Walker.
\newblock Cyclic {A}dams operations.
\newblock {\em J. Pure Appl. Algebra}, 2016.
\newblock In press. Preprint available at arXiv:1601.05072.

\bibitem{BrunsHerzog}
Winfried Bruns and J{\"u}rgen Herzog.
\newblock {\em Cohen-{M}acaulay rings}, volume~39 of {\em Cambridge Studies in
  Advanced Mathematics}.
\newblock Cambridge University Press, Cambridge, 1993.

\bibitem{BE}
David~A. Buchsbaum and David Eisenbud.
\newblock Algebra structures for finite free resolutions, and some structure
  theorems for ideals of codimension {$3$}.
\newblock {\em Amer. J. Math.}, 99(3):447--485, 1977.

\bibitem{EGA}
Jean Dieudonn{\'e} and Alexander Grothendieck.
\newblock \'{E}l\'ements de g\'eom\'etrie alg\'ebrique.
\newblock {\em Inst. Hautes \'Etudes Sci. Publ. Math.}, 4, 8, 11, 17, 20, 24,
  28, 32, 1961--1967.

\bibitem{GS87}
H.~Gillet and C.~Soul\'e.
\newblock Intersection theory using {A}dams operations.
\newblock {\em Inventiones Mathematicae}, 90:243--277, 1987.

\bibitem{Ha}
Robin Hartshorne.
\newblock Algebraic vector bundles on projective spaces: a problem list.
\newblock {\em Topology}, 18(2):117--128, 1979.

\bibitem{KuranoRoberts}
Kazuhiko Kurano and Paul~C. Roberts.
\newblock Adams operations, localized {C}hern characters, and the positivity of
  {D}utta multiplicity in characteristic {$0$}.
\newblock {\em Trans. Amer. Math. Soc.}, 352(7):3103--3116, 2000.

\bibitem{PS}
C.~Peskine and L.~Szpiro.
\newblock Dimension projective finie et cohomologie locale. {A}pplications \`a
  la d\'emonstration de conjectures de {M}. {A}uslander, {H}. {B}ass et {A}.
  {G}rothendieck.
\newblock {\em Inst. Hautes \'Etudes Sci. Publ. Math.}, (42):47--119, 1973.

\bibitem{RobertsBook}
Paul~C. Roberts.
\newblock {\em Multiplicities and {C}hern classes in local algebra}, volume 133
  of {\em Cambridge Tracts in Mathematics}.
\newblock Cambridge University Press, Cambridge, 1998.

\end{thebibliography}

\end{document}